\documentclass[11pt, a4paper]{amsart}
\title[Conformal restriction and unitarizing measures for Virasoro]{Infinitesimal conformal restriction and unitarizing measures for Virasoro algebra}
\usepackage{mathtools}
\usepackage{lmodern}
\usepackage{amsthm,amsmath,amsfonts,amssymb}
\usepackage{graphicx}
\usepackage[shortlabels]{enumitem}
\usepackage{cite,url}
\usepackage{bbm}
\usepackage{comment}

\usepackage{multirow}
\usepackage{array}
\newcolumntype{P}[1]{>{\centering\arraybackslash}p{#1}}
\usepackage{longtable,booktabs}
\usepackage{hyperref}
\hypersetup{
    colorlinks=true, 
    linkcolor=black, 
    urlcolor=black, 
    citecolor=black,
    linktoc=all 
}

\allowdisplaybreaks

\let\OLDthebibliography\thebibliography
\renewcommand\thebibliography[1]{
  \OLDthebibliography{#1}
  \setlength{\parskip}{1pt}
  \setlength{\itemsep}{2pt}
}

\newtheorem{thm}{Theorem}[section]
\newtheorem{cor}[thm]{Corollary}

\newtheorem{lem}[thm]{Lemma}
\newtheorem{prop}[thm]{Proposition}

\theoremstyle{definition} 
\newtheorem{df}[thm]{Definition}

\newtheorem{remark}[thm]{Remark}

\newcommand{\Rb}{\mathbb{R}}
\newcommand{\Cb}{\mathbb{C}}

\usepackage[font=normal]{caption}
\usepackage{floatrow}

\usepackage{mathtools}

\usepackage[dvipsnames]{xcolor}

\global\long\def\ii{\mathfrak{i}}
\global\long\def\ee{\mathrm{e}}

\renewcommand{\liminf}{\varliminf}
\renewcommand{\limsup}{\varlimsup}

\newcommand{\norm}[1]{\lVert #1 \rVert}

\newcommand{\mc}[1]{\mathcal{#1}}
\newcommand{\m}[1]{\mathbb{#1}}

\renewcommand\Re{\operatorname{Re}}
\renewcommand\Im{\operatorname{Im}}

\def\PSL{\operatorname{PSL}}
\def\SLE{\operatorname{SLE}}

\def\mob{\mathrm{M\ddot{o}b}}
\def\Diff{\operatorname{Diff}}

\def\supp{\operatorname{supp}}

\def\g{\gamma}

\def\d{\delta}

\def\l{\lambda}
\def\k{\kappa}

\def\o{\omega}
\def\O{\Omega}
\def\vare{\varepsilon}

\def\Chat{\hat{\m{C}}}

\def\dd{\mathrm{d}}

\def\1{\mathbf{1}}

\newcommand{\enu}{{\vare \nu}}

\newcommand{\Loop}{\mathit{Loop}}

\newcommand{\Div}{\mathrm{Div}}

\usepackage[dvipsnames]{xcolor}

\def \1{\mathbbm{1}}

\begin{document}

\author[Maria Gordina]{Maria Gordina}
\address{Department of Mathematics\\
University of Connecticut\\
Storrs, CT 06269,  U.S.A.}
\email{\protect\url{maria.gordina@uconn.edu}}

\author[Wei Qian]{Wei Qian}
\address{City University of Hong Kong --
Department of Mathematics and Department of Physics (on leave from CNRS, Universit\'e Paris-Saclay, Laboratoire de Math\'ematiques d'Orsay)}
\email{weiqian@cityu.edu.hk}

\author[Yilin Wang]{Yilin Wang}
\address{Institut des Hautes Etudes Scientifiques \\
Bures-sur-Yvette, 91440, France}
\email{\protect\url{yilin@ihes.fr}}



\begin{abstract}
We use the SLE$_\kappa$ loop measure to construct a natural representation of the Virasoro algebra of central charge $c = c(\kappa) \le 1$. In particular, we introduce a non-degenerate bilinear Hermitian form (and non positive-definite) using the SLE loop measure and show that the representation is indefinite unitary. Our proof relies on the infinitesimal conformal restriction property of the SLE loop measure.

\bigskip

\noindent {R{\tiny ÉSUMÉ}.} Nous utilisons  la mesure sur les lacets de SLE$_\kappa$ pour construire une repr\'esentation naturelle de l’alg\`ebre de Virasoro de charge centrale $c = c(\kappa) \le 1$. En particulier, nous introduisons une forme hermitienne non-dégénérée (et non définie positive), en utilisant la mesure sur les lacets de SLE, et nous montrons que la représentation est unitaire indéfinie. Notre preuve repose sur la propriété de restriction conforme infinitésimale de la mesure sur les lacets de SLE.
\end{abstract}

\maketitle

\section{Introduction}

In \cite{AiraultMalliavin2001}, Airault and Malliavin asked if one can find a probability measure on the group of diffeomorphisms of the unit circle $\Diff(S^1)$ such that the associated $L^2$ space contains a closed subspace of holomorphic functionals on which the Virasoro algebra acts unitarily via Kirillov's representation. If they exist, such measures are called \emph{unitarizing measures}. Airault, Malliavin, and Thalmaier later answered the question negatively in  \cite{AiraultMalliavinThalmaier2002}. 
Finding modifications of unitarizing measures has then inspired many works, e.g. \cite{Airault_affine,airault2003realization,AiraultMalliavinThalmaier2004,AiraultMalliavinThalmaier2010,AiraultNeretin2008,Fang2002,Kontsevich_SLE,malliavin_heat}. 

Circle homeomorphisms can be identified with Jordan curves on the Riemann sphere $\Chat = \m C \cup\{\infty\}$ via conformal welding. Hence, measures on the family of circle homeomorphisms can also be described as measures on the family of Jordan curves.
In \cite{Kontsevich_SLE}, Kontsevich and Suhov postulated a few axioms on loop measures without constructing examples, but they hinted at links to both SLEs and unitarizing measures. These axioms are closely related to the conformal restriction covariance of chordal SLE outlined by Lawler, Schramm, and Werner in \cite{LSW_CR_chordal} (see also \cite{MR2178043}).
Now, explicit examples of measures satisfying these axioms (often referred to as Malliavin--Kontsevich--Suhov measures) have been constructed by Zhan in \cite{zhan2020sleloop}, called the \emph{SLE loop measure}. It is a one-parameter family of $\sigma$-finite measures $\mu^c_{\Chat}$ indexed by $0 < \k \le 4$ on Jordan curves, where the parameter space is in one-to-one correspondence with the central charge $c = c (\k) := (6-\k) (3\k - 8)/(2\k) \in (-\infty, 1]$. However, the precise relation of \cite{Kontsevich_SLE} to the unitarizing measures was more allusive. 

In the spirit of the Berezin quantization, already mentioned in \cite{AiraultMalliavin2001} (see also \cite{AST}), one candidate of the unitarizing measure is a formal measure of the form $\ee^{\l I} \dd m$, where $\l$ is a multiple of the central charge, $I$ is a K\"ahler potential for the unique right-invariant K\"ahler metric on the homogeneous space of circle diffeomorphisms $\mob (S^1)\backslash \Diff(S^1)$, namely the Weil--Petersson metric, and $\dd m$ is the ``Weil--Petersson volume form''. Since $\Diff (S^1)$ is infinite-dimensional, the volume form does not make literal sense. 

On the other hand, \cite{W2} proved that the aforementioned K\"ahler potential (whose explicit expression is given in \cite{TT06}) equals the Loewner energy $I^L(\cdot)$ for Jordan curves introduced by Rohde and the last author \cite{RW}.  
Moreover, Carfagnini and the last author showed  in \cite{carfagnini2023onsager} that the Loewner energy is the action functional of the SLE loop measure.
Roughly speaking, they defined a certain $\vare$-neighborhood $B_\vare (\g)$ of an analytic Jordan curve $\g$, such that for all $0< \k \le 4$, 
$$\lim_{\vare \to 0}\frac{\mu^c_{\Chat} (B_\vare(\g))}{\mu^c_{\Chat} (B_\vare(S^1))} = \exp \left(\frac{c}{24} I^L(\g)\right).$$
This result shows that the SLE loop measure may be viewed as the measure $\ee^{(c/24)I^L} \dd m$ and should be relevant to the unitarizing measure.

The goal of this work is to construct a Virasoro representation using the SLE loop measure. 
We summarize our construction below while highlighting a few differences with the original proposal by Airault and Malliavin. 
\begin{itemize}
    \item Instead of $\Diff(S^1)$, we consider the space of Jordan curves in $\m C$ separating $0$ and $\infty$, denoted by $\Loop$.
    \item We let $\mu^c$ be the measure of SLE$_\k$ loops restricted to $\Loop$, where $c = c (\k) \in (-\infty, 1]$. The measure $\mu^c$ is infinite and $\sigma$-finite (instead of being a probability measure).
    \item Let $\tau : z \mapsto 1/\bar z$. We introduce the bilinear form
    $$(F,G)_{\mu^c} : = \int_\Loop \overline {F(\g)}\, G (\tau (\g)) \, \dd \mu^c (\g)$$
    for $F, G \in L^2_\tau (\Loop, \mu^c)$, the $L^2$ space with respect to $\mu^c$ intertwined with $\tau$, as defined in Definition~\ref{df_L_2}.
    This bilinear form is Hermitian, non-degenerate, but not positive-definite. See Proposition~\ref{p.BilinearForm}.
  \item The holomorphic vector fields $v_k = - z^{k+1} \dd/\dd z$  and $\ii v_k = -\ii z^{k+1} \dd/\dd z$ act on $\m C \setminus \{0\}$. They induce an action of the complex Witt algebra generated by $ (L_k^{\m C})_{k \in \m Z}$ on $L^2_\tau (\Loop, \mu^c)$. We use this representation of the Witt algebra instead of Kirillov's representation. We note that this representation
  is also used in  \cite{ ChavezPickrell2014,ChavezPhDThesis2015} for SLE$_{8/3}$ loop measure where $c = 0$.
  \item We define explicitly a central extension $\mc V_c$ (the complex Virasoro algebra) generated by densely defined closable operators $\{\mc L_k\}_{k \in \m Z} \cup \{\mc K\}$ acting on $L^2_\tau (\Loop, \mu^c)$. The action of $\mc L_k$ is modified from that of $L^{\m C}_k$, and $\mc K$ acts as the identity map. See Proposition~\ref{prop:Phi_example} and Corollary~\ref{cor:vir_gen}.
\end{itemize}

Our main result is the following.
\begin{thm}[See Theorem~\ref{thm:adjoint}]
The adjoint of the Virasoro generator $\mc L_k$ with respect to $(\cdot, \cdot)_{\mu^c}$ is $\mc L_{-k}$ for all $k \in \m Z$. 
\end{thm}

One important property of $\mu^c$ is that, if we compare it with its pushforward under a conformal map, the Radon--Nikodym derivative has a precise form (see Lemma~\ref{lem:CR}).
The key ingredient of our construction of the Virasoro representation is the following result describing the infinitesimal conformal variation of the SLE loop measure. 
\begin{prop} [See Proposition~\ref{prop:divergence}] \label{prop:intro_divergence}
Let $\mu^{c}$ be the SLE$_\k$ loop measure of central charge $c$.  
    We have
    $$ \Div_{\mu^c} \left(L_k^{\m C}\right) =    \begin{cases}
     - \frac{c}{12}   P_{-k}, \quad &\text{for } k \le 0, \\
     \frac{c}{12}  Q_{k}, \quad & \text{for } k \ge 0.
    \end{cases} $$
    where $\Div_{\mu^c}$ is the divergence  (see Definition~\ref{df:divergence}) with respect to the measure $\mu^c$. Here,  $P_k$ and $Q_k$ for $k \ge 0$ are the Neretin polynomials associated with $f^{-1}$ and $g^{-1}$ (see Definition~\ref{df:neretin}) and $f$ (resp., $g$) is a conformal map from $\m D$ to the bounded connected component $\O$ (resp., the unbounded connected component $\O^*$) sending $0$ to $0$ (resp., to $\infty$).
\end{prop}
This result was also stated in \cite[Sec.\,2.5.2]{Kontsevich_SLE} by Kontsevich and Suhov but using the Neretin polynomials associated with $f$ and $g$ which are different from our $P_k$ and $Q_k$. 
We provide a detailed proof of this result in Section~\ref{sec:infini}. This result also implies the integration by parts formula for the SLE loop measure (Corollary~\ref{cor:IBP}).

Representations of the Virasoro algebra play a prominent role in the Liouville conformal field theory,
starting from the foundational work \cite{BELAVIN1984333}. Relations between SLE, CFT, and Virasoro algebra have also been studied in numerous works, such as \cite{bauer2002slekappa,bauer2003conformal,bauer2004conformal,BBK,Friedrich_Kalkkinen,Friedrich_Werner_03,kangMakarov,Peltola}. More recently, in \cite{Baverez_vir_LCFT}, the Hamiltonian of Liouville conformal field theory (LCFT) is diagonalized via the action of the Virasoro algebra. A complete characterization of the algebraic structure of LCFT has been achieved through the construction of highest-weight representations of the Virasoro algebra \cite{baverez2023irreducibility}.
The works \cite{Quantum_zipper,AHS-EJP914} showed that the SLE$_\kappa$ loop measure can be constructed by conformally welding two Liouville quantum gravity disks. Hence, it will be interesting to compare the Virasoro action we obtain with that in LCFT.
A recent paper \cite{maibach2024} connects the breaking of conformal symmetry in (euclidean) 2D CFT to a conformal anomaly as expressed by the conformal restriction property in planar random geometry. This leads to a nontrivial central extension of the classical conformal symmetry described as the Virasoro algebra.

While this is not the subject of this paper, we mention other constructions of measures on $\Diff(S^1)$ or the space of Jordan curves. In addition to the study by Airault, Malliavin, and Thalmaier in \cite{AiraultMalliavinThalmaier2010}, the Brownian motion on $\operatorname{Diff}\left( S^{1} \right)$ has been constructed and studied in \cite{AiraultMalliavin2000, AiraultMalliavin2001, AiraultMalliavinThalmaier2004, Fang2002,GordinaWu2008,  WuMang2011}. This direction was inspired by the earlier work by Kirillov in \cite{Kirillov1998a}, including Riemannian geometry of such infinite-dimensional spaces. While the unitarizing measures have not been constructed in the papers by Airault, Malliavin, Thalmaier, et al., there were a number of results relating such a measure to representations of the Virasoro algebra in  \cite{airault2003realization, AiraultMalliavin2001, AiraultMalliavinThalmaier2002, AiraultNeretin2008}, and reviewed in \cite{Lescot2007a}. This is the direction that is very relevant to our results and connects it to the influential work in representation theory by Kirillov--Neretin--Yuriev in \cite{KirillovNeretinBook1994, KirillovYuriev1988}. 
In \cite{TK}, Amaba and Yoshikawa also investigated stochastic Loewner--Kufarev evolution on the loop space and obtained an integration by parts formula that is reminiscent of our Proposition~\ref{prop:intro_divergence}. We also mention constructions of random circle homeomorphisms using the conformal welding of random surfaces \cite{AJKS,binder2023,KMS_welding} and the Malliavin--Shavgulidze measure, see, e.g., \cite{bauerschmidt2024sch,Malliavin_Malliavin,Shavgulidze97}.

\section{Preliminaries}

\subsection{Brownian loop measure}

The Brownian loop measure in $\Chat = \m C \cup \{\infty\}$ was first introduced by Lawler and Werner in \cite{LW2004loupsoup}. For a domain $D \subset \Chat$ and $A_1, A_2 \subset D$, we denote by $\mc B_D(A_1, A_2)$ the mass of Brownian loops that are contained in $D$ and intersect both $A_1$ and $A_2$. The outer boundary of a Brownian loop $\d$ (i.e., the boundary of the connected component of $\Chat \setminus \d$ containing $\infty$) induces a measure $\mc W$, introduced by Werner\cite{werner_measure},  which is supported on simple loops that have the local geometry of an $\SLE_{8/3}$ curve (hence also called the \emph{SLE$_{8/3}$ loop measure}).

Let $K_1, K_2$ be two disjoint compact subsets of $\Chat$. The total mass of the Brownian loop measure on $\Chat$ intersecting both $K_1$ and $K_2$ is infinite. There are two different renormalizations to obtain a finite mass of loops hitting $K_1$ and $K_2$. In \cite{MR3038676}, the renormalized Brownian loop measure is defined to be
\begin{align*}
\Lambda^\ast(K_1, K_2) := \lim_{r\to 0} \mc B_{\Chat \setminus B(z, r)} (K_1, K_2) - \log \log r^{-1},
\end{align*}
where $B(z,r)$ is the Euclidean ball centered at $z$ with radius $r$.
It was shown in \cite{MR3038676} that this definition does not depend on the choice of $z \in \Chat \setminus (K_1 \cup K_2)$.
An alternative way is to consider the mass $\mc W(K_1, K_2)$ of loops in $\Chat$ intersecting both $K_1$ and $K_2$ under the measure $\mc W$, which is proved to be finite in \cite{KemppainenWerner2016}.

\subsection{SLE loop measure}

The definition of SLE$_{8/3}$ loop measure can be extended to any Riemann surface by Werner's work \cite{werner_measure}.  For other values of $\k$, Benoist and Dub\'edat \cite{Benoist_loop} constructed the SLE$_2$ loop measure.
Later in \cite{zhan2020sleloop}, Zhan constructed the SLE$_\kappa$ loop measures for all $\kappa\in (0,8)$. For $\kappa\in(0,4]$, SLE$_\kappa$ loop measures are examples of \emph{Malliavin--Kontsevich--Suhov} loop measure with central charge $c (\kappa) = (6-\kappa)(3\kappa-8)/(2\kappa) \in (-\infty, 1]$ that we now formulate in terms of the conformal restriction property as follows.  

\begin{df}[Malliavin--Kontsevich--Suhov measure]\label{d.MKSmeasure} Let $c \in (-\infty, 1]$, we call a family of measures $(\mu^c_D)_{D \subset \Chat}$ on the set of \emph{simple} loops contained in a domain $D \subset \Chat$ a \emph{Malliavin--Kontsevich--Suhov measure} of central charge $c$ if it satisfies the following properties.

\begin{itemize}
    \item (RC) \textit{Restriction covariance property} : suppose $D \subset \Chat$, then the Radon--Nikodym derivative is given by 
\begin{equation}\label{eqn.conformal.restr}
\frac{\dd \mu_{D}^c}{ \dd \mu^c_{\Chat}} \left( \cdot \right) = \mathbbm{1}_{ \left\{ \cdot \, \subset D \right\} } \exp \left( \frac{c}{2}\, \Lambda^{\ast} \left( \cdot, \Chat \setminus D \right) \right).
\end{equation}

\item  (CI) \textit{Conformal invariance}: if $D$ and $D^{\prime}$ are conformally equivalent domains in the plane, the pushforward of $\mu^c_{D}$ via any conformal map from $D$ to $D^{\prime}$, is exactly  $\mu^c_{D^{\prime}}$.
\item ($\sigma$-finiteness):  For all $0 < r <1 $, let $\m A_r := \{z \in \m C \, |\, r<|z| < 1/r\}$ and
$$
\Loop^r : = \{\g \text { Jordan curve in } \m C \text{ separating } 0 \text { and } \infty \,|\, \g \subset \m A_r\},
$$
then we have 
$\mu^c_{\Chat} \{\Loop^r\} < \infty$.
\end{itemize}
\end{df}

This definition is a rewriting of the property of having a trivial section in the determinant line bundle on the space of simple loops as postulated in \cite{Kontsevich_SLE} by Kontsevich and Suhov while alluding to the works of Malliavin. 
 When $c = 0$, the uniqueness of the MKS measure up to a multiplicative constant was proved in \cite{werner_measure} (prior to \cite{Kontsevich_SLE}) and is given by the SLE$_{8/3}$ loop measure $\mc W$. For other values of $c$, the uniqueness is shown in the very recent work \cite{Baverez_Jego}.

The following lemma is a consequence of Definition~\ref{d.MKSmeasure}. 

\begin{lem}[Conformal restriction covariance]\label{lem:CR}
Let $\varphi : D \to D'$ be a conformal map between two domains in $\Chat$. Then
$$\frac{\dd  \varphi_*(\mu^c_{\Chat} \1_{\{\cdot \subset D\}})}{\dd \mu^c_{\Chat} \1_{\{\cdot \subset D'\}}} (\g) = \exp \left(\frac{c}{2} \left( \Lambda^* (\g, \Chat \setminus D')- \Lambda^*(\varphi^{-1} (\g), \Chat \setminus D) \right)\right),$$
where $\varphi_*\mu$ denotes the pushforward of the measure $\mu$ under $\varphi$.
\end{lem}
\begin{proof}
By \eqref{eqn.conformal.restr}, we have
\begin{align*}
    \dd\mu^c_{\Chat}\1_{\{\cdot \subset D\}}=\exp \left(- \frac{c}{2}\, \Lambda^{\ast} \left( \cdot, \Chat \setminus D \right) \right) \dd \mu^c_D(\cdot).
\end{align*}
Applying $\varphi_*$ on both sides, and noting that (CI) implies $\varphi_*(\mu^c_D)=\mu^c_{D'}$, we have
\begin{align}\label{eq:1}
\dd\varphi_*\left(\mu^c_{\Chat}\1_{\{\cdot \subset D\}}\right) =\exp \left(- \frac{c}{2}\, \Lambda^{\ast} \left( \varphi^{-1}(\cdot), \Chat \setminus D \right) \right) \dd \mu^c_{D'}(\cdot).
\end{align}
Again by \eqref{eqn.conformal.restr}, we have
\begin{align}\label{eq:2}
\dd \mu^c_{D'}(\cdot)=\dd\mu^c_{\Chat}\1_{\{\cdot \subset {D'}\}}\exp \left( \frac{c}{2}\, \Lambda^{\ast} \left( \cdot, \Chat \setminus D' \right) \right).
\end{align}
Combining \eqref{eq:1} and \eqref{eq:2} completes the proof.
\end{proof}

It is known that, although $\Lambda^* (K_1, K_2)$ is not equal to $\mc W(K_1, K_2)$ \cite[Lem.\,2.4]{carfagnini2023onsager}, we may still replace the difference of $\Lambda^*$ in Lemma~\ref{lem:CR} by that of $\mc W$ by the following theorem. 

\begin{thm}[See {\cite[Thm.\,2.5]{carfagnini2023onsager}} and the remark after it]\label{thm:Lambda_W_equal}
If $D$ and $D'$ are conformally equivalent simply connected or doubly connected open subsets of $\Chat$ and $\varphi: D \to D'$ is a conformal map. Then for all Jordan curves $\g \subset D'$, we have
$$ \Lambda^*(\g, \Chat \setminus D') -\Lambda^* (\varphi^{-1}(\g), \Chat \setminus D) =   \mc W(\g, \Chat \setminus D')-\mc W (\varphi^{-1}(\g), \Chat \setminus D).$$
\end{thm}

\begin{df}\label{df:loop}
We denote by 
$$\Loop := \cup_{0<r<1} \Loop^r = \{\text{Jordan curves in } \m C \text { separating } 0 \text{ and } \infty\}.$$
For each $\g\in\Loop$, let $\O$ and $\O^{\ast}$ be the connected components of $\Chat \setminus \g$ containing $0$ and $\infty$ respectively. Let $f$ be the unique conformal map $\m D \to \O$ with $f(0)=0$ and $f'(0)>0$. 
We say $f$ is the \emph{normalized conformal map} associated with $\g \in \Loop$.

We can then identify $\Loop$ with the space of conformal maps $f: \m D \to \O$ with $f(0)=0$ and $f'(0)>0$, and further endow $\Loop$ with the Carath\'eodory topology, which is the topology of uniform convergence on this space of conformal maps. See also Remark~\ref{rem:topology}.
We define
\begin{equation}
\mu^c := \mu^c_{\Chat} \mathbbm{1}_{\Loop}.
\end{equation}
\end{df}

In the sequel, we will only consider $\mu^c$ to be the measure induced from the SLE$_\k$ loop measure. (Given the claim in \cite{Baverez_Jego} about the uniqueness of $\mu^c$, it coincides with the SLE$_\k$ loop measure up to a multiplicative constant.)

\subsection{Involution on the space of loops}\label{s.Involution}

\begin{df}
We denote by $\tau$ the map $z \mapsto 1/\bar z$, which induces a continuous involution on the spaces $\Loop$ and $\Loop^r$ for all $ 0 < r <1$. 
\end{df}

\begin{lem}\label{lem:tau}
For $\kappa\in (0,4]$, the SLE$_\kappa$ loop measure $\mu^c$ on $\Loop$ is invariant under $\tau$.
\end{lem}
\begin{proof}
As the SLE$_\kappa$ loop measure is invariant under conformal maps, we only need to show that it is also invariant under the complex conjugation $z\mapsto \bar z$.

In \cite[Thm.\,4.2]{zhan2020sleloop}, Zhan constructed the SLE$_\kappa$ loop measure $\mu^c_{\Chat}$ by integrating a two-sided whole-plane SLE$_\kappa$ measure $\nu^\#_{z\rightleftharpoons w}$ between two points $z$ and $w$ in $\Chat$:
\begin{align*}
\mu^c_{\Chat}=\mathrm{Cont}(\cdot)^{-2}\int_\Cb\int_\Cb\nu^\#_{z\rightleftharpoons w} G_\Cb(w-z) dw dz,
\end{align*}
where $\mathrm{Cont}(\gamma)$ stands for the $1+\kappa/8$-dimensional Minkowski content of a curve $\gamma$, and $G_\Cb(z)=|z|^{-2(1-\kappa/8)}$ is Green's function. Since the Minkowski content and Green's function are invariant under conjugation, it suffices to show that the image of $\nu^\#_{z\rightleftharpoons w}$ under conjugation is equal to $\nu^\#_{\bar z\rightleftharpoons \bar w}$.

Recall that $\nu^\#_{z\rightleftharpoons w}$ is constructed by first running a whole-plane SLE$_\kappa(2)$ $\nu^\#_{z\to w}$ from $z$ to $w$ and then running a chordal SLE$_\kappa$ curve from $z$ to $w$ in the remaining domain. By conformal invariance of SLEs, it suffices to consider the case $z=0$ and $w=\infty$.
A whole-plane Loewner chain $(K_t)_{-\infty<t<\infty}$ growing from $0$ to $\infty$ is generated by the Loewner equation
\begin{align}
\partial_t g_t(z)=g_t(z)\frac{e^{-\ii \lambda_t}+ g_t(z)}{e^{-\ii\lambda_t}-g_t(z)}, \quad \lim_{t\to-\infty} e^t g_t(z)=z,
\end{align}
where $\lambda: (-\infty, \infty)\to \Rb$ is the driving function, and $g_t$ is the conformal map from $\Chat \setminus K_t$ to $\Chat \setminus \m D$ that sends $\infty$ to $\infty$ with $g_t'(\infty)=\lim_{z\to \infty} z/g_t(z)>0$. The driving function $\lambda_t$ of $\nu^\#_{0\to \infty}$, as given by \cite[Eq.\,(3.2),(3.3)]{zhan2020sleloop}, is a random process whose law is the same as $-\lambda_t$. This implies that the law of $\nu^\#_{0\to \infty}$ is invariant under $z\mapsto \bar z$.  
Similarly, by the chordal Loewner equation and noting that the driving function $W_t$ of a chordal SLE$_\kappa$ in $\m{H}$ has the same law as $-W_t$, we know that a chordal SLE$_\kappa$ in $\m{H}$ from $0$ to $\infty$ is invariant under $z\mapsto - \bar z$.

Let $\gamma_1$ be a curve with law $\nu^\#_{0\to \infty}$ and $\gamma_2$ be a chordal SLE$_\kappa$ in $\Chat\setminus \gamma_1$ from $0$ to $\infty$. Note that
if $f$ is a conformal map from $\Chat\setminus \gamma_1$ onto $\m H$ that leaves $0, \infty$ fixed, then $z\mapsto - \overline{f(\bar z)}$ is a conformal map from $\Chat\setminus \overline{\gamma_1}$ onto $\m H$ that leaves $0, \infty$ fixed. Combined with the previous paragraph, we can conclude that $(\overline{\gamma_1}, \overline{ \gamma_2})$ has the same distribution as $(\gamma_1, \gamma_2)$, which completes the proof.
\end{proof}

\begin{cor}\label{cor:tau_inv_r}
   For all $ 0 < r <1 $, the measure $\mu^c$ restricted to $\Loop^r$ is also invariant under $\tau$.
\end{cor}

\subsection{Function spaces}

We consider the following bilinear form.  For Borel measurable functions $F, G : \Loop \to \m C$, let
\begin{equation}\label{eq:bilinear}
(F, G)_{\mu^c} :=  \int_{\Loop} \overline{F (\g)} \,\tau^* G(\g) \,\dd \mu^c (\g) 
\end{equation}
whenever the integral on the right-hand side exists. Here $\tau^* G (\g): = G \circ \tau (\g)$ and $\tau$ is the involution on $\Loop$ introduced in Section~\ref{s.Involution}.

The involution $\tau$ induces the following polarization on functions: 
\begin{align*}
&  \mc F^+ := \{F: \Loop \longrightarrow \mathbb{C}\, | \, F = \tau^* F\},
\\
& \mc F^- := \{F: \Loop \longrightarrow \mathbb{C}\, | \, F = - \tau^* F\}. 
\end{align*}
For any function $F: \Loop \longrightarrow \mathbb{C}$,  we can decompose $F=F_{+}+F_{-}$, where $F_+= (F+ \tau^* F)/2$ and $F_-=(F-\tau^* F)/2$, so that $F_{+} \in  \mc F^+$, and $F_{-} \in  \mc F^-$.

\begin{df} \label{df_L_2}
We define $
L_{\tau}^{2}\left( \Loop,  \mu^{c}\right)$ as the space of measurable functions $F:\Loop \longrightarrow \mathbb{C}$ such that $(F_{+}, F_{+})_{\mu^c}< \infty$ and $-(F_{-}, F_{-})_{\mu^c}< \infty$, and then
\[
\Vert F \Vert_{L_{\tau}^{2}}^{2}:=(F_{+}, F_{+})_{\mu^c}+(F_{-}, F_{-})_{\mu^c}.
\]
By $L^2(\Loop, \mu^c) := \{F : \Loop \to \m C \,|\, \int_{\Loop} |F|^2 \,\dd \mu^c < \infty\}$ we denote the standard $L^2$-space.
\end{df}

\begin{prop}\label{p.BilinearForm}
The bilinear form $(\cdot, \cdot)_{\mu^c}$ is Hermitian and non-degenerate.  Moreover, $L^2(\Loop, \mu^c) = L^2_\tau(\Loop, \mu^c)$ as  sets.
\end{prop}

\begin{proof}

Since $\mu$ is invariant under $\tau$ by Lemma~\ref{lem:tau}, we have
$$(F,G)_{\mu^c} = \int_{\Loop} \overline F \, \tau^* G \, \dd \mu^c = \int_{\Loop} \overline{\tau^* F} \, G \,\dd \tau^*\mu = \overline{\int_{\Loop} \tau^* F \, \overline G \,\dd \mu^c } = \overline{(G,F)_{\mu^c}}. $$
This shows that $(\cdot,\cdot)_{\mu^c}$ is Hermitian.

It is clear that the bilinear form \eqref{eq:bilinear} is positive definite on $\mc F^+ \times \mc F^+$ and negative definite on $\mc F^- \times \mc F^-$. Moreover, we have
\begin{align*}
& (F_{+}, F_{-})_{\mu^c} = \int_{\Loop} \overline {F_{+}} \, \tau^* F_{-} \, \dd \mu^c = -\int_{\Loop} \overline {F_{+}} \, F_{-} \,\dd \mu^{c}
\\
= & \overline{(F_{-}, F_{+})_{\mu^c}} = \overline{\int_{\Loop} \overline {F_{-}} \, \tau^* F_{+} \,\dd \mu^c } 
 = 
\overline{\int_{\Loop}  \overline {F_{-}} \, F_{+} \, \dd \mu^c } =
\int_{\Loop}  \overline{F_{+}} \,  F_{-} \,\dd \mu^c,
\end{align*}
therefore $ (F_{+}, F_{-})_{\mu^c} =0$. This shows that $(\cdot,\cdot)_{\mu^c}$ is non-degenerate.

The inclusion $L^2_\tau(\Loop, \mu^c) \subset L^2(\Loop, \mu^c)$ is clear since for all $F \in L^2_\tau$, $F = F_+ + F_-$ with $F_+, F_- \in L^2$ by assumption. Conversely, if $F \in L^2$, then Lemma~\ref{lem:tau} shows that $\tau^* F \in L^2$. This implies $F \in L^2_\tau$.
\end{proof}

We equip $\Loop$ with a coordinate system as in \cite[Sec.\,5]{ChavezPickrell2014}. 
Let $f$ be the normalized conformal map in Definition~\ref{df:loop}, then we have the following expansion 
\begin{align*}
f(z)= \ee^{\rho_0} \, z\left(1+ \sum_{n\ge 1} u_n z^n\right). 
\end{align*}
The map $\g \mapsto (\rho_0, u_1, u_2, \ldots )$ defines a coordinate system on $\Loop$. Recall that the quantity $\rho_0 =\rho_0(\g)$ is also called the \emph{log conformal radius} of $\g$.

\begin{df}[Cylinder functions]\label{df.cylinder_fun}
Let $\mathcal{C}^{\infty}_{c}$ be the space of \emph{compactly supported smooth cylinder functions} on $\Loop$ in the coordinates introduced above, namely, there exists $k \ge 1$ such that
\[
F=h\left( \rho_0 \right)e\left( u_{1}, \ldots, u_{k}, \overline{u_{1}}, \ldots, \overline{u_{k}} \right),
\]
where $h \in C_{c}^{\infty}\left(  \mathbb{R}
\right)$ and $e \in C_{c}^{\infty}\left( \mathbb{R}^{2k} \right)$  are compactly supported smooth functions. 

We also write $\mc C^\infty$ for the space of \emph{smooth cylinder functions} where we only assume that $h$ and $e$ are smooth. 
\end{df}

\begin{remark}[Topology]\label{rem:topology}
    The product topology induced by the coordinate system above coincides with the Carath\'eodory topology, namely, uniform convergence on finitely many coordinates is equivalent to the uniform convergence of the normalized conformal map $f$ on all compact subsets of $\m D$. This can be seen from the Bieberbach--de Branges theorem \cite{deBranges} (which shows that $|u_n| \le n+1$ and controls the tail of the power series). Conversely, uniform convergence of $f$ on compact implies convergence of $\rho_0$ and $u_k$ by the Cauchy integral formula.

    We also remark that as $|u_n| \le n+1$ implies that if $e \in C^\infty (\m R^{2k})$ is induced from a function on $\Loop$, then $e$ is automatically in $C_c^\infty (\m R^{2k})$.
\end{remark}

\begin{lem}[Compactness]\label{lem:compact}
    Let $-\infty < a \le b < \infty$. The closure $F_{a,b}$ of the family of normalized conformal maps $f$ for loops $\{\rho_0 \in [a,b]\} \subset \Loop$ in the space of univalent functions is compact for the Carath\'eodory topology.
\end{lem}
\begin{proof}
     We note that when $a = b = 0$, $F_{0,0}$ is the family of Schlicht functions and it is a classical result that $F_{0,0}$ is compact. Therefore, as $F_{a,b}$ is homeomorphic to $[a,b] \times F_{0,0}$, we also obtain that $F_{a,b}$ is compact.
\end{proof}

We also have the following elementary lemma.

\begin{lem}\label{l.ConfRadius}
    For $r\in(0,1)$, the conformal radius $\ee^{\rho_0}$ of any loop in $\Loop^r$ is bounded from above by $1/r$ and below by $r$.
\end{lem}
\begin{proof}
This follows directly from the Schwarz Lemma.
\end{proof}

\begin{remark}[Disintegration of the loop measure]\label{r.Disintegration}
The scaling invariance of $\mu^c$ shows that the measure induced on the coordinates is invariant under the translation map $\rho_0 \mapsto \rho_0 + t$ for any $t \in \m R$. 
Additionally, the $\sigma$-finiteness implies that 
there exists a probability measure $\mathbf{P}^c$ on the loops with conformal radius $1$ (or equivalently, on $F_{0,0}$), and 
$\l>0$ such that we can disintegrate $\mu^c$ as $\l\, \dd \rho_0  \otimes \mathbf{P}^c$, where $\dd \rho_0$ is the Lebesgue measure on the first coordinate. For this, we used the fact that $F_{a,b}$ is compact for the Carath\'eodory topology  (by Lemma~\ref{lem:compact}), then the disintegration for probability measures on compact 
spaces follows from \cite[p.~78~III]{DellacherieMeyerBook1978}, the fact the SLE loop measure is Borel by \cite[Prop.~B.1]{Baverez_Jego}
\end{remark}

\begin{prop}\label{p.Denseness}
The space $\mathcal{C}_{c}^{\infty}$ is dense in $L^{2}\left(\Loop, \mu^c \right)$.
\end{prop}

\begin{proof} 
First, we can use the disintegration of the measure in Remark~\ref{r.Disintegration} to see that the Borel $\sigma$-algebra on $\Loop$ is generated by the products of (Borel) measurable sets in $\Loop$ with conformal radius $1$ and (Borel) measurable sets of finite measure on $\mathbb{R}$. This allows us to reduce the questions of the denseness of cylinder functions in $L^{2}$ spaces with respect to the two measures in disintegration. Then the denseness in $L^{2}\left(\Loop, \mu^c \right)$ follows from Dynkin's $\pi$-$\lambda$ (monotone class) theorem for the product measure space, e.g. \cite[Prop.~24]{SchwartzL1974-1975LectureNotes} for $\sigma$-finite measures. 

Note that the first component (depending on the conformal radius) in the definition of functions in  $\mathcal{C}_{c}^{\infty}$ is an $L^{2}$-function for $L^{2}$ with respect to a weighted Lebesgue measure on $[0, \infty)$. Such functions can be approximated by continuous functions with compact support. Then observe that $C^{\infty}_{c}\left( \mathbb{R} \right)$ is dense in $C_{0}\left(  \mathbb{R} \right)$ with respect to the topology of compact convergence by the Stone-Weierstra\ss~theorem.
Since $C_{c}\left( \mathbb{R} \right) \subset C_{0}\left(\mathbb{R} \right)$, then  $C^{\infty}_{c}$ is dense in $C_{c}$ as well.

Approximation of $L^{2}$-functions over the loops in $\Loop$ with conformal radius $1$ (equivalently,  over $F_{0,0}$) with respect to the measure $\mathbf{P}^c$ uses the fact that $\mathbf{P}^c$ is a probability  measure, therefore $C_{c}^{\infty}$ functions are dense in $L^{2}$. We refer for more details to \cite[Thm. 22.8 (Density Theorem), Prop. 28.23]{DriverToolsBook2004}. Note that this is applicable in our setting, as by \cite[Prop.~B.1]{Baverez_Jego} the SLE loop measure is a Borel measure with respect to the Carath\'{e}odory topology, and then we can use the fact that Polish spaces are Radon spaces, e.g. \cite[Sect.~IX.3]{BourbakiIntegrationII}. 
\end{proof}

\section{Infinitesimal conformal restriction}\label{sec:infini}

The goal of this section is to derive the infinitesimal form of the conformal restriction covariance property of the SLE loop measure Proposition~\ref{prop:intro_divergence} (Proposition~\ref{prop:divergence}). This was claimed in \cite[Sec.\,2.5.2]{Kontsevich_SLE} with a slightly different result.  For this, we need to first understand the infinitesimal variation of the Brownian loop measure. 

\subsection{Variations of Brownian loop measure}
We now describe the setup of the variational formula that we will consider. 

For $\g \in \Loop$,  let $\O$ and $\O^*$ be the connected components of $\Chat \setminus \g$ containing $0$ and $\infty$ respectively, and $\nu \in L^\infty (\m C)$ be a Beltrami differential with compact support in $\Chat \setminus\gamma$.  For $\vare \in \m{R}$ with $\|\vare \nu\|_{\infty} < 1$, let $\omega^\enu: \Chat \to \Chat$ be any quasiconformal mapping solving the Beltrami equation
$$\frac{\bar \partial \o^\enu}{\partial \o^\enu} = \enu.$$
In particular, $\o^\enu$ is conformal in a neighborhood of $\g$. By the measurable Riemann mapping theorem, another solution to the Beltrami equation is of the form $m \circ \o^\enu$ where $m \in \PSL (2,\m C)$ is an M\"obius map of $\Chat$.
We let $\g^{\enu} =\omega^\enu(\g)$. 

On the other hand, by fixing a normalization of $\o^\enu$, we can make the map $\o^\enu$ smooth in $\vare$ in a neighborhood of $0$. Then, we may also describe the variation of the loop in terms of the vector field 
$v_\nu : = \dd \o^{\enu}/ \dd \vare |_{\vare = 0}$, and  
$\nu = \bar \partial v_{\nu}$. 
From this, it is not hard to see that for any vector field $v$ that is holomorphic in a neighborhood of $\g$, there is  $\nu \in L^\infty (\m C)$ with compact support in $\Chat \setminus\gamma$ such that $v = v_\nu$.

We will only consider the case when $\nu$ is supported on one side of $\g$ and aim to compute the value of 
\begin{equation}\label{eq:W_var_goal}
\frac{\dd}{\dd \vare}\bigg|_{\vare=0} \mc W (\gamma^\enu, \Chat \setminus \o^{\enu} (D)) = \frac{\dd}{\dd \vare}\bigg|_{\vare=0} \Lambda^* (\gamma^\enu, \Chat \setminus \o^{\enu} (D)) 
\end{equation}
for any domain simply connected domain $D$ containing $\g$ such that $D \cap \textnormal{supp}(\nu) = \varnothing$ (the equality follows from Theorem~\ref{thm:Lambda_W_equal}). 
Note first that 
$$ \mc W (\gamma^\enu, \Chat \setminus \o^{\enu} (D))= \mc W (m (\gamma^\enu), \Chat \setminus m(\o^{\enu} (D)))$$
as Werner's measure is $\PSL(2,\m C)$ invariant. Hence, the variational formula \eqref{eq:W_var_goal} does not depend on the choice of solution to the Beltrami equation and will be described in both the Beltrami differential $\nu$ and holomorphic vector field $v$.

\begin{prop}\label{prop:var_loop_general}
     Let $\g$ be a Jordan curve, $\nu$ be an infinitesimal Beltrami differential with compact support in $\O$, and $v = v_\nu$. Let $D$ denote a simply connected domain in $\Chat$ containing $\g$ and $D \cap \supp (\nu) = \varnothing$. We have
      \begin{align} \label{eq:variation_Bloop}
    \frac{\dd}{\dd \vare}\bigg|_{\vare=0} \Lambda^\ast (\gamma^\enu, \Chat \setminus \o^{\enu} (D)) & =   \frac{1}{3\pi} \, \mathrm{Re}  \int_{\O} \nu(z) \mc S [f^{-1}](z) \,\dd^2 z  \nonumber\\ 
     & = \frac{1}{3\pi} \, \mathrm{Re} \left[ \frac{1}{2\ii} \int_{\partial \O} v(z) \mc S [f^{-1}](z) \,\dd z\right],
\end{align}
where $\mc S[\varphi] = \varphi'''/\varphi' - (3/2)(\varphi''/\varphi')^2$ is the Schwarzian derivative, $\dd^2 z$ is the Euclidean area measure, $\dd z$ is the contour integral, and $f$ is any conformal map $\m D \to \O$.

If $\nu$ has compact support in $\O^*$ and $D$ is a simply connected domain containing $\g$ and $D \cap \supp (\nu) = \varnothing$, then we have
\begin{align} \label{eq:variation_Bloop_g}
    \frac{\dd}{\dd \vare}\bigg|_{\vare=0} \Lambda^\ast (\gamma^\enu, \Chat \setminus \o^{\enu} (D)) & =   \frac{1}{3\pi} \, \mathrm{Re}  \int_{\O^*} \nu(z) \mc S [g^{-1}](z) \,\dd^2 z  \nonumber\\ 
    &= \frac{1}{3\pi} \, \mathrm{Re} \left[ -\frac{1}{2\ii} \int_{\partial \O} v(z) \mc S [g^{-1}](z) \,\dd z\right],
\end{align}
where $g$ is any conformal map $\m D \to \O^*$.
\end{prop}
\begin{remark}\label{rem:change_domain}
We note that if $\g \subset D' \subset D$, it follows from the definition of $\Lambda^\ast$ that 
$$ \Lambda^\ast (\gamma, \Chat \setminus D') =  \Lambda^\ast (\gamma, \Chat \setminus D) + \mc B_{D} (\g, D \setminus D'). $$
In particular, if we vary $D$ by a conformal map $\o^\enu$ (since $D \cap \supp (\nu) = \emptyset$), the second term on the right-hand side is invariant by conformal invariance of the Brownian loop measure. Hence, the left-hand sides of \eqref{eq:variation_Bloop} and \eqref{eq:variation_Bloop_g} do not depend on $D$ as long as $D \cap \supp (\nu) = \emptyset$.
\end{remark}
\begin{remark}\label{rem:indep_choice}
We note that by the chain rule of the Schwarzian derivative:
$$\mc S [\varphi \circ \psi] = \mc S [\varphi] \circ \psi (\psi')^2 + \mc S[\psi] $$
and the fact that $\mc S [m] = 0$ for all M\"obius maps $m$,  $\mc S[f^{-1}]$ does not depend on the choice of $f$ since any other choice is of the form $f \circ m$ where $m$ is a M\"obius map preserving $\m D$.
\end{remark}

\begin{remark}\label{rem:mu_v}
The identity on the right-hand side of \eqref{eq:variation_Bloop}  follows from 
$\nu = \bar \partial  v$ and  Stokes' formula:
\begin{align*}
     \int_{\O} \nu(z) \mc S [f^{-1}](z) \,\dd^2 z   & = \int_{\O \cap \supp(\nu)} \nu(z) \mc S [f^{-1}](z) \,\dd^2 z \\
     & = \frac{1}{2 \ii} \int_{\O \cap \supp(\nu)} \bar \partial \left(v(z) \mc S [f^{-1}](z)\right) \,\dd \bar z \wedge \dd z \\
     & = \frac{1}{2 \ii} \int_{\O \cap \supp(\nu)} \dd \left( v(z) \mc S [f^{-1}](z) \, \dd z \right) \\
     & = \frac{1}{2 \ii} \int_{\partial \O}  v(z) \mc S [f^{-1}](z) \, \dd z.
\end{align*}
  Therefore, when $\g$ is not smooth, the contour integral on $\partial \O$ is understood as the contour integral along a smooth curve in $\O$ homotopic to $\partial \O$ outside of the support of $\nu$ and similarly for Equation \eqref{eq:variation_Bloop_g}.
\end{remark}

The proof of Proposition~\ref{prop:var_loop_general} follows from the following result with exactly the same setup but assuming that $\g$ is Weil--Petersson. The main reason for this additional assumption is that the proof of Theorem~\ref{thm:WP} goes through a variational formula of the Loewner energy \cite{W2,W3}, which is finite only for Weil--Petersson quasicircles.

\begin{thm}[See {\cite[Cor.\,1.6]{SungWang}}]\label{thm:WP} Let $\g$ be a Weil--Petersson quasicircle. Let $\nu\in L^\infty(\m{C})$ be an infinitesimal Beltrami differential with compact support in $\Chat \setminus\gamma$. For $\vare \in \m{R}$ with $\|\vare \nu\|_{\infty} < 1$, let $\omega^\enu: \Chat \to \Chat$ be any quasiconformal mapping with Beltrami coefficient $\enu$.
Let $\g^{\vare \nu} = \omega^{\vare \nu} (\g)$.
For any annulus $A$ containing $\g$ such that $A \cap \textnormal{supp}(\nu) = \varnothing$,
we have
    \begin{align} \label{eq:variation_loop_measure_1}
    & \frac{\dd}{\dd \vare}\bigg|_{\vare=0} \mc W (\gamma^\enu, \Chat \setminus \o^{\enu} (A)) \nonumber \\
     & = \frac{1}{3\pi} \, \mathrm{Re} \left[ \int_{\O} \nu(z) \mc S [f^{-1}](z) \,\dd^2 z+ \int_{\O^*} \nu(z) \mc S [g^{-1}](z) \,\dd^2 z \right].
\end{align}
\end{thm}
If $\nu$ is supported on one side of $\g$, Remark~\ref{rem:change_domain} shows that we can take $A$ to be a simply connected domain $D$ as in Proposition~\ref{prop:var_loop_general}.

\begin{proof}[Proof of Proposition~\ref{prop:var_loop_general}]
   When $\g$ is a Weil--Petersson quasicircle, then Proposition~\ref{prop:var_loop_general} follows from Theorem~\ref{thm:WP} and Remark~\ref{rem:mu_v}. 

   Now we assume $\g$ is a general Jordan curve and show \eqref{eq:variation_Bloop}. The proof of \eqref{eq:variation_Bloop_g} is exactly the same. 
   Without loss of generality, let $\g \in \Loop$ (and it separates $0$ and $\infty$). Let $f : \m D \to \O$ be a conformal map fixing $0$ and $g: \m D \to \O^*$ be a conformal map fixing $\infty$. Let $\nu \in L^\infty (\O)$ be compactly supported. 
   Let $D$ be a simply connected domain such that $\g \subset D$ and $D \cap \supp (\nu) = \emptyset$.
   Let $(\g_r)_{r \in [1-\d, 1+\d] \setminus \{1\}}$ be a family of analytic curves defined by 
   $$\g_r = f (r S^1), \quad \text{if } r < 1, \text{ and } \g_r = g (r S^1) \quad  \text{if } r > 1,$$
   and $\d$ is small enough such that $\g_r \subset D$ for all $r \in (1-\d, 1)$ (note that $\g_r$ is automatically in $D$ for $r > 1$). We also write $\g = \g_1$.

Now, we consider the continuous function 
$$ L (t, r) =  \mc W (\o^{t\nu} (\gamma_r), \Chat \setminus \o^{t\nu} (D)), \quad \text{ for all } (t, r) \in [0,1/\norm{\nu}_\infty) \times (1-\d, 1+ \d).$$ 

Theorem~\ref{thm:WP} applied to $\g_{t,r} : = \o^{t\nu} (\g_r)$ gives that for all $r \neq 1$,
  \begin{align*}
      \partial_t L(t,r) & =   \frac{\dd}{\dd \vare}\bigg|_{\vare=0} \mc W (\o_t^\vare (\gamma_{t,r}), \Chat \setminus \o_t^\vare( \o^{t\nu} (D)))\\
      & = \frac{1}{3\pi} \, \mathrm{Re} \left[ \int_{\O_{t,r}} \nu_t(z) \mc S [f_{t,r}^{-1}](z) \,\dd^2 z\right]
  \end{align*}
where $f_{t,r}$ is a conformal map from $\m D$ onto $\O_{t,r}$ the bounded connected component of $\m C \setminus \g_{t,r}$,  the quasiconformal map $\o_t^\vare := \o^{(t+\vare) \nu} \circ (\o^{t\nu})^{-1}$ has Beltrami coefficient 
$$ \frac{\bar \partial \o_t^\vare}{\partial \o_t^\vare} = \left(\frac{\vare \nu}{ 1- t(t+\vare) |\nu|^2 } \frac{\partial \o^{t\nu}}{\overline{\partial\o^{t\nu}}}\right) \circ (\o^{t\nu})^{-1}$$
and 
$$\nu_t = \frac{\dd}{\dd \vare}\bigg|_{\vare=0} \left(  \frac{\bar \partial \o_t^\vare}{\partial \o_t^\vare}\right) = \left(\frac{\nu}{ 1- t^2|\nu|^2 } \frac{\partial \o^{t\nu}}{\overline{\partial\o^{t\nu}}}\right) \circ (\o^{t\nu})^{-1}$$
is continuous in $t$ and independent of $r$. 

It is clear that $(t,r) \to \g_{t,r}$ is continuous for the Carath\'eodory topology, which implies that $(t,r) \mapsto \mc S[f_{t,r}]$ is continuous for the topology of uniform convergence on compact subsets of $\m D$.  
From this, we see that $\partial_t L (t,r)$ extends to a continuous function $h$ on $[0,1/\norm{\nu}_\infty) \times (1-\d, 1+ \d)$ with $h (t,r) = \partial_t L(t,r)$ for all $r \neq 1$.

On the other hand,  since $D$ is simply connected, we have the monotonicity 
$$ L(t,r) > L(t,1) > L(t, r') \quad \text{ for all } t \text{ and } r  < 1 < r'. $$
Hence, let $r = r(t) < 1$ such that $L(0,r(t)) - L(0,1) = o (t)$, we have
   $$\frac{L(t, 1) - L(0,1)}{t} \le \frac{L(t, r) - L(0,r)}{t} +  \frac{L(0, r) - L(0,1)}{t} \xrightarrow[]{t \to 0} h(0,1).$$
   This implies 
   $$\limsup_{t\to 0} \frac{L(t, 1) - L(0,1)}{t} \le h (0,1). $$
   Similarly, we also have the lower bound by considering $r = r(t) > 1$:
  $$ \liminf_{t \to 0} \frac{L(t, 1) - L(0,1)}{t}  \ge h (0,1).$$ 
 This implies that $t \mapsto L(t,1)$ is also smooth and has derivative at $0$ given by 
  $$
    \frac{\dd}{\dd \vare}\bigg|_{\vare=0} \mc W (\gamma^\enu, \Chat \setminus \o^{\enu} (D)) =  h(0,1) = \frac{1}{3\pi} \, \mathrm{Re}  \int_{\O} \nu(z) \mc S [f^{-1}](z) \,\dd^2 z $$
    which completes the proof by \eqref{eq:W_var_goal}.
\end{proof}

\subsection{Neretin polynomials}

\begin{df}\label{df:neretin}
Let $\g \in \Loop$, 
$f : \m D \to \O$ and $g : \m D \to \O^*$ be any conformal maps. 
    We define the \emph{Neretin polynomials} $P_k = P_k [\g] \in \m C$ as the coefficients in the series expansion of $z^2 \mc S [f^{-1}]$, namely,
$$z^2 \mc S [f^{-1}] (z) = \sum_{ k \ge 0} P_k z^k.$$
We note that by Remark~\ref{rem:indep_choice}, $\mc S[f^{-1}]$ only depends on $\O$.
Similarly, define $Q_k = Q_k [\g] \in \m C$ as the coefficients in the series expansion of $z^{2} \mc S [g^{-1}]$, namely,
$$z^{2} \mc S [g^{-1}] (z) = \sum_{ k \ge 0} Q_k z^{-k}.$$
\end{df}

We collect the properties of $P_k$ in the following lemma which explains the name  of ``polynomial''.
\begin{lem}\label{lem:Neretin_property} Let $f$ be the normalized conformal map associated with $\g \in \Loop$.
    For all $k \ge 0$, the function $P_k$ is a polynomial in the coefficients of $f^{-1}$. For $\lambda >0$ and $\g \in \Loop$, let  $\lambda \g$ be the curve obtained as the image of $\g$ by $z \mapsto \lambda z$. We have 
    \begin{equation}\label{eq:scale_neretin}
       P_k[\l\g] = P_k [\g] \l^{-k}.
\end{equation}
Moreover, for all $r \in (0,1)$, $P_k$ is bounded on $\Loop^r$ and we have $P_k \in \mc C^\infty$ (Definition~\ref{df.cylinder_fun}). 
\end{lem}
\begin{proof}
   We show \eqref{eq:scale_neretin} first. Let $f_\l = \l f: \m D \to \l\O$ be the normalized conformal map associated with $\l \g$. Then by the chain rule, we obtain
$$
\sum_{k\ge 2} P_k [\l \g] z^{k-2} = \mc S[f_\l^{-1}] (z) = S [f^{-1}] (\l^{-1} z) \l^{-2}  = \sum_{k\ge 2} P_k [\g]  \l^{-k} z^{k-2}.
$$
This proves \eqref{eq:scale_neretin}. Alternatively, we can also see it from $L_0 P_k  = k P_k$. 
   
   Recall that we wrote
   $$f(z) = \ee^{\rho_0 (\g)} z \left(1 + \sum_{n\ge 1} u_n z^n \right).$$
   It is convenient to also introduce the coefficients of $f^{-1}$ and write
   $$f^{-1} (w) = \ee^{-\rho_0 (\g)} w \left(1 + \sum_{k\ge 1} v_k w^k \right). $$ 
   From \eqref{eq:scale_neretin}, we assume that $\rho_0 (\g)= 0$ from now on. 

   It was shown in \cite[p.~742]{Kirillov1998a} that the coefficient of $z^k$ of $z^2 \mc S[f](z)$ is a polynomial in $(u_1, \ldots, u_k)$. It is also true that $P_k[\g]$ is also a polynomial in $(v_1, \ldots, v_k)$ as the proof is algebraic and only depends on the local expansion of the conformal maps at $0$.

   Now we show that $P_k$ is cylindrical in the coefficients $(u_n)$. It suffices to relate the coefficients $(v_n)$ to $(u_n)$. Indeed, 
   from the identity
   \begin{align*}
      z  & =  f^{-1} (f (z))  = f^{-1}\Bigg(z\Big(1 +\sum_{n\ge 1} u_{n} z^{n}\Big) \Bigg) \\
      & = z \Big (1 + \sum_{n\ge 1} u_n z^n\Big)\Bigg(1 + \sum_{k\ge 1} v_k z^k \Big (1 + \sum_{n\ge 1} u_n z^n\Big)^k \Bigg)
   \end{align*}
and by comparing the coefficients of the expansion on both sides, we can deduce inductively that $v_n$ is a rational function in $(u_1, \ldots, u_n)$, which is known as the Lagrange--B\"urmann formula. 

 Finally, we show that $P_k$ is bounded on the space of loops where $\rho_0 (\g) = 0$.  This implies the boundedness of $P_k$ on $\Loop^r$ by Lemma~\ref{l.ConfRadius} and \eqref{eq:scale_neretin}. It also implies that $P_k \in \mc C^\infty$ as it is a bounded cylindrical rational function. 
 
 For this, Koebe's one-quarter theorem shows that the bounded connected component $\O$ of $\m C\setminus \g$ contains the ball $B(0,1/4)$. Therefore, the map $h : z\mapsto 4 f^{-1} (z/4)$ is univalent on $\m D$ and satisfies $h'(0) = 1$.  By Bieberbach--de Branges theorem \cite{deBranges}, 
 $$h (w) = w \Bigg( 1 + \sum_{k \ge 1} v_k 4^{-k} w^k \Bigg)$$
 has coefficients $|v_k 4^{-k}| \le k+1$. This shows that $|v_k| \le (k+1) 4^k$ and as $P_k$ is a polynomial in $(v_1, \ldots, v_k)$, we obtain that $P_k$ is also bounded.  
\end{proof}

The next result shows that the functions $P_k$ and $Q_k$ are related by $\tau : z \mapsto 1/\bar z$.
\begin{lem}\label{lem:P_Q_tau}
 For all $k \geqslant 0$,  we have $Q_k \circ \tau = \overline{P_k}$. In particular, $Q_k$ is bounded on $\Loop^r$ for all $r \in (0,1)$.
\end{lem}
\begin{proof}
    Let $\g \in \Loop$ and $f$ and $g$ be conformal maps associated with $\g$. Let $\sigma:z\mapsto \bar z$ be the conjugation map. Then the map $\tilde g := \tau \circ f \circ \sigma$ is a conformal map $\m D \to \tilde \O^*$ where $\tilde \O^*$ is the unbounded connected component of $\m C \setminus \tau (\g)$.
    Therefore, from the definition of $Q_k$,
    $$ z^2 \mc S[\tilde g^{-1}] (z) = \sum_{k \ge 0} Q_k[\tau (\g)] z^{-k}. $$
    On the other hand,
    $$ z^2 \mc S[\tilde g^{-1}] (z) = z^2 \overline{\mc S[f^{-1}] \circ \tau (\bar \partial \tau)^2} = z^2 \sum_{k \ge 0} \overline{ P_k (\g)} \left(\frac{1}{z}\right)^{k-2} \frac{1}{z^4} =   \sum_{k \ge 0} \overline{P_k (\g)} z^{-k}.$$
    Identifying the coefficients, we obtain the identity.
\end{proof}

\subsection{Action of Witt algebra and infinitesimal conformal restriction}\label{s.ConformalRestr}

We now apply an infinitesimal variation of the Brownian loop measure using a specific basis of vector fields $v_{k} =  - z^{k + 1} \frac{\dd}{\dd z}$ on $\Chat$, where $ k \in \m Z$. The Lie bracket is given by 
$$[v_n, v_k] = (n - k) v_{n + k}.$$
The vector field $v_k$ is holomorphic in $\Chat \setminus B(0,\d)$ for all $\d > 0$ if $k \le 1$, and holomorphic in $B (0,R)$ for all $R > 0$ if $k \ge -1$.
The flow generated by the vector field $v_k$ is the solution to  
$$\frac{\dd \varphi_{k,t} (z)}{ \dd t} = v_k (\varphi_{k,t} (z)), \qquad \varphi_0 (z) = z.$$
We also consider the vector field $\ii v_k = - \ii z^{k+1} \frac{\dd}{\dd z}$, with the flow denoted by $\psi_{k,t}$.
Explicitly, $ \varphi_{k,t}$ and  $\psi_{k,t}$ are given by
\begin{equation}\label{eq:flow}
\varphi_{k,t} (z) = \frac{z}{(1 + k t z^k)^{1/k}},\qquad \psi_{k,t} (z) = \frac{z}{(1 + k \ii t z^k)^{1/k}}
\end{equation}
when $k \neq 0$ and 
\begin{equation}\label{eq:flow_0}
\varphi_{0,t} (z) = \ee^{-t} z,\qquad \psi_{0,t} (z) = \ee^{- \ii t} z,
\end{equation}
which are biholomorphic on $B(0,R)$ when $k \ge -1$ and on $\Chat \setminus B (0,\d)$ when $k \le 1$ for small enough $t$.

The following result follows from applying Proposition~\ref{prop:var_loop_general} to $v_n$ and $\ii v_n$.
\begin{cor} \label{cor:infinitesimal_BML}
Let $\g \in \Loop$. For $n \le 0$, $\d >0$ such that $B (0,\d) \subset \O$, we have
\begin{align*}
    \frac{\dd}{\dd t}\bigg|_{t=0} \Lambda^\ast \big(\varphi_{n,t}(\g), \varphi_{n,t}(B (0,\d))\big) & = -\frac{\Re P_{-n}}{3} \\
    \frac{\dd}{\dd t}\bigg|_{t=0} \Lambda^\ast \big(\psi_{n,t}(\g), \psi_{n,t}(B (0,\d))\big) & = \frac{\Im P_{-n}}{3}. 
\end{align*}

For $n \ge 0$, $R >0$ such that $\O \subset B (0,R)$, we have
\begin{align*}
    \frac{\dd}{\dd t}\bigg|_{t=0} \Lambda^\ast \big(\varphi_{n,t}(\g), \Chat \setminus \varphi_{n,t}(B (0,R))\big) & = \frac{\Re Q_n}{3} \\
    \frac{\dd}{\dd t}\bigg|_{t=0} \Lambda^\ast \big(\psi_{n,t}(\g), \Chat \setminus \psi_{n,t}(B (0,R))\big) & = -\frac{\Im Q_n}{3}. 
\end{align*}
\end{cor}
\begin{remark}
    Note that from the definition $P_0 = P_1 = Q_0 = Q_1 = 0$. This is consistent with the fact that $v_1, v_0, v_{-1}$ generates the M\"obius transformations of $\Chat$ and the Brownian loop measure is invariant under such transformations.
\end{remark}
\begin{proof}
For $n \le 0$, $v_n$ can be realized by an infinitesimal Beltrami $\nu$ supported in a neighborhood of $0$, hence we may apply \eqref{eq:variation_Bloop} and obtain
    \begin{align*}
     \frac{\dd}{\dd t}\bigg|_{t=0} \Lambda^\ast \big(\varphi_{n,t}(\g), \varphi_{n,t}(B (0,\d))\big) & = \frac{1}{3\pi} \, \Re \left[ \frac{1}{2\ii} \int_{\partial  B (0,\d)} v_n(z) \mc S [f^{-1}](z) \,\dd z\right]\\
    & = \frac{1}{3\pi} \, \Re \left[ \frac{1}{2\ii} \int_{\partial  B (0,\d)} - z^{n- 1} \sum_{ k \ge 0} P_k z^k \,\dd z\right] \\
    & =-  \frac{\Re (P_{-n})}{3},\\
   \frac{\dd}{\dd t}\bigg|_{t=0} \Lambda^\ast \big(\psi_{n,t}(\g), \psi_{n,t}(B (0,\d))\big)& = \frac{1}{3\pi} \, \Re \left[ \frac{1}{2\ii} \int_{\partial  B (0,\d)} \ii v_n(z) \mc S [f^{-1}](z) \,\dd z\right]\\
    & = \frac{1}{3\pi} \, \Re \left[ \frac{1}{2\ii} \int_{\partial  B (0,\d)} - \ii z^{n- 1} \sum_{ k \ge 0} P_k z^k \,\dd z\right] \\
    & =  \frac{\Im (P_{-n})}{3}. 
\end{align*}

For $n \ge 0$, we apply \eqref{eq:variation_Bloop_g} and obtain
 \begin{align*}
 \frac{\dd}{\dd t}\bigg|_{t=0} \!\!\! \Lambda^\ast \big(\varphi_{n,t}(\g), \Chat \setminus \varphi_{n,t}(B (0,R))\big) & = \frac{1}{3\pi} \, \Re \left[ \frac{1}{2\ii} \int_{\partial  B (0,R)}\!\!\! - v_n(z) \mc S [g^{-1}](z) \,\dd z\right]\\
    & = \frac{1}{3\pi} \, \Re \left[ \frac{1}{2\ii} \int_{\partial  B (0,R)}  z^{n- 1} \sum_{ k \ge 0} Q_k z^{-k} \,\dd z\right]
    \\
    &   = \frac{\Re (Q_n)}{3}.
\end{align*}
Similarly, we have
 \begin{align*}
     \frac{\dd}{\dd t}\bigg|_{t=0} \Lambda^\ast \big(\psi_{n,t}(\g), \Chat \setminus \psi_{n,t}(B (0,R))\big) 
    & = -\frac{\Im (Q_n)}{3}
    \end{align*}
    as claimed.
\end{proof}

Let $L_k$ be the real vector field on $\Loop$ induced by $v_k$ which also acts on the space of suitable functions on $\Loop$, and $L'_k$ as the real vector field $\ii v_k$. 

We have the following commutation relations
\begin{align*}  
[L_k, L_m] &= (k-m) L_{k+m}, \\
[L_k, L_m'] & =  (k-m) L'_{k+m},\\
[L'_k, L'_m] & = - (k-m) L_{k+m}.
\end{align*}

Define $L_k^{\m C}$ to be the complexified vector field 
on $\m C \setminus \{0\}$  (which  acts on functions on $\Loop$)
\begin{equation}\label{eq:df_Witt}
   L_k^{\m C}  =  \frac12 (L_k - \ii L'_k). 
\end{equation}
Recall that we introduced the space of cylinder functions $C_c^\infty$ in Definition~\ref{df.cylinder_fun}. Then for $F\in C_c^\infty$, we have
\begin{align}\label{e.LoopDerivative}
L_k^{\m C}F (\g) = \frac12\frac{\dd}{\dd t}\Big|_{t = 0} \Big(F (\varphi_{k,t} (\g)) - \ii F (\psi_{k,t} (\g))\Big). 
\end{align}
Note that \eqref{e.LoopDerivative} is similar to \cite[Eq.\,(2.0.5.3)]{ChavezPhDThesis2015}. 
The family $(L_k^\m C)_{k \in \m Z}$ is a basis of the complex Witt algebra. Namely,
$$[L_n^{\m C}, L_k^{\m C}] = (n - k) L_{n+k}^{\m C}.$$

\begin{prop}\label{p.Stability} Let $k \in \m Z$, and $L = L_k, L_k'$, or $L^{\m C}_k$, then we have  $L \left( \mathcal{C}^{\infty}_{c} \right) \subseteq \mathcal{C}^{\infty}_{c}$. 
\end{prop}
\begin{proof}
   The action of $L_k$ and $L_k'$ on the coordinates $u_n$ and $\rho_0$ is derived explicitly in \cite[Prop.\,2.1.6]{ChavezPhDThesis2015} using the variational formula of Duren and Schiffer \cite{DurenSchiffer1962}.   The expressions are algebraic, which implies the stability of $\mathcal{C}^{\infty}_{c}$. 
\end{proof}
This shows that $(L_k^{\m C})_{k\in \m Z}$ is a representation of the Witt algebra by densely defined operators on $L^2(\Loop,\mu^c)$ by Proposition~\ref{p.Denseness}.

\begin{df}\label{df:divergence}
For a real vector field $L$ and a Borel measure $\mu$ on $\Loop$, the \emph{divergence} of $L$ with respect to $\mu$ is a function 
$\Div_{\mu} (L) : \Loop \to \m R$ such that for all smooth and bounded cylinder functions $F : \Loop \to \m C$ supported on $\Loop^r$ for some $r \in (0,1)$, we have
$$ \int_{\Loop} L F (\g) \dd \mu (\g) = \frac{\dd}{\dd t}\Big |_{t = 0} \int_{\Loop} F (\varphi_t (\g)) \dd \mu (\g) = \int_{\Loop} \Div_{\mu} (L) (\g) F(\g) \dd \mu (\g) $$
where $\varphi_t =\ee^{tL}$ is the flow induced by $L$.
We extend the definition for complex vector fields by complex linearity.
\end{df}

The divergence of a vector field also gives the following integration by parts formula.
\begin{cor}[Integration by parts formula] \label{cor:IBP}
The chain rule implies that for a real vector field $L$ and smooth cylinder functions $F$ and $G$ as above
\begin{equation}\label{eq:div_two_f}
    \int_{\Loop} F (\g) L G (\g) \dd \mu (\g) = \int_{\Loop} \big(- L F (\g) + \Div_{\mu} (L) (\g) F (\g)\big) G (\g) \dd \mu (\g).
\end{equation}
For a complex vector field $L^{\m C}$,
\begin{equation}\label{eq:div_two_f_c}
    \int_{\Loop} F (\g) L^{\m C} G (\g) \dd \mu (\g) = \int_{\Loop} \left(- L^{\m C} F (\g) + \Div_{\mu} (L^{\m C}) (\g) F (\g)\right) G (\g) \dd \mu (\g).
\end{equation}
\end{cor}

\begin{cor}\label{cor:weight_div}
If $\mu ' = g \mu \ll \mu$, where $g$ is a smooth density function, then 
$$\Div_{\mu'} (L) = - \frac{Lg}{g} + \Div_{\mu} (L).$$
\end{cor}
\begin{proof}
    It follows immediately from \eqref{eq:div_two_f}.
\end{proof}

Finally, combining previous results we obtain the main proposition of this section.
\begin{prop} \label{prop:divergence}
Let $\mu^{c}$ be a Malliavin--Kontsevich--Suhov measure for the central charge $c$.  
    We have
    $$ \Div_{\mu^c} \left(L_n^{\m C}\right) =    \begin{cases}
     - \frac{c}{12}   P_{-n}, \quad &\text{for } n \le 0 \\
     \frac{c}{12}  Q_{n}, \quad & \text{for } n \ge 0.
    \end{cases} $$
\end{prop}

\begin{proof}
We consider first $n \ge 0$.
Let $F$ be a smooth function $\Loop \to \m C$ supported on loops in $\Loop^r$, in particular, contained in $B(0,R)$ for any $R > 1/r$. As for small $t$, $\varphi_t := \varphi_{n,t}$ is conformal in $D = B(0,R)$. Let $D_t = \varphi_t^{-1} (B(0,R))$ we have
\begin{align*}
   & \int_{\g \in \Loop} F (\varphi_t (\g)) \dd \mu^c (\g) 
     =  \int_{\g \subset D_t} F (\varphi_t (\g)) \dd \mu^c (\g) \\
   & \overset{\text{``RC''} }{=} \int_{\g \subset D_t} F (\varphi_t (\g)) \exp \left(- \frac{c}{2} \Lambda^\ast (\g, \Chat \setminus D_t)\right)\dd \mu^c_{D_t} (\g) \\
    & \overset{\text{``CI''} }{=} \int_{\eta \subset B(0,R)} F (\eta) \exp \left(- \frac{c}{2} \Lambda^\ast (\varphi_t^{-1} (\eta), \Chat \setminus D_t)\right)\dd \mu^c_{B(0,R)} (\eta) \\
   & \overset{\text{``RC''} }{=}   \int_{\eta \subset B(0,R)} F (\eta) \exp \left(- \frac{c}{2} \left(\Lambda^\ast (\varphi_{-t} (\eta), \Chat \setminus D_t) - \Lambda^\ast (\eta, \Chat \setminus D_0) \right)\right)\dd \mu^c (\eta).
\end{align*}
Taking derivative in $t$ and evaluating at $t = 0$, we obtain that by Corollary~\ref{cor:infinitesimal_BML}
$$ \int_{\g \in \Loop} L_n F (\g) \dd \mu^c  = \int_{\g \in \Loop} \frac{c \Re Q_n}{6}  F (\g) \dd \mu^c (\g). $$
We note that the right-hand side is well-defined since $Q_n$ is bounded on $\Loop^r$ by Lemma~\ref{lem:P_Q_tau}.
Similarly, we have
$$ \int_{\g \in \Loop} L_n' F (\g) \dd \mu^c  = \int_{\g \in \Loop} \frac{-c \Im Q_n}{6}  F (\g) \dd \mu^c (\g). $$
Since it holds for all $r \in (0,1)$ and all $F$  supported on $\Loop^r$, we obtain that  for $n \ge 0$,
$$\Div_{\mu^c} \left(L_n^{\m C}\right) = \frac{c}{12} Q_n.$$

Similarly for $n \le 0$, we have
$$ \int_{\g \in \Loop} L_n F (\g) \ \dd \mu^c  = \int_{\g \in \Loop} \frac{- c \Re P_{-n}}{6}  F (\g) \ \dd \mu^c (\g) $$
and
$$ \int_{\g \in \Loop} L_n' F (\g) \ \dd \mu^c  = \int_{\g \in \Loop} \frac{c \Im P_{-n}}{6}  F (\g) \ \dd \mu^c (\g) $$
which gives
    $$\Div_{\mu^c} \left(L_n^{\m C}\right) =  - \frac{c}{12} P_{-n}$$
as claimed.
\end{proof}

\section{Unitarizing measure for Virasoro algebra}

In this section, we complete the construction of the representation of Virasoro algebra using the SLE loop measure (Corollary~\ref{cor:vir_gen}) and show that it is (indefinite) unitary (Theorem~\ref{thm:adjoint}).

\subsection{Virasoro representations}

We have constructed the representation \eqref{eq:df_Witt} of the Witt algebra $(L_n^\m C)_{n \in \m Z}$ in Section~\ref{s.ConformalRestr}.
The Virasoro algebra $\mc {V}_c$ with the central charge $c$ is a central extension of the Witt algebra by a central element $\mc K$, so that $\mc {V}_c$ is spanned by $\{\mc L_n\}_{n \in \m Z} \cup \{\mc K\}$ which satisfy the commutation relation 

$$
[\mc L_n, \mc L_k] = (n - k) \mc L_{n + k} + \frac{c}{12} (n^3 - n) \delta_{n, -k} \mc {K}, \qquad [\mc L_n, \mc K] = 0.
$$
We can obtain a representation of the Virasoro algebra on a space of functions on $\Loop$ by modifying the representation of the Witt algebra. We recall the standard proof in the next lemma.

With a slight abuse of notation, we denote by $\Phi$ the multiplication operator $M_{\Phi}$ by a measurable function $\Phi$ whose restriction to $\Loop^r$ is bounded for all $ 0< r < 1$. Note that these are densely defined (unbounded) operators with $M_{\Phi}^{\dag}=M_{\tau^{\ast}\overline{\Phi}}$. Here, $\dag$ means the adjoint with respect to the bilinear form \eqref{eq:bilinear}:
$$(F,G)_{\mu^c} = \int_{\Loop} \overline{F(\g)} \tau^* G (\g) \, \dd \mu^c (\g).$$

\begin{lem} \label{lem:vir_cond}
Let $\mc K$ act as the identity map on the space of functions on the space $\Loop$ and  $\mc L_n := L_n^{\m C} + \Phi_n$,
where $\Phi_n : \Loop \to \m C$ is a smooth cylinder function in $\mc C^\infty$ that is bounded on $\Loop^r$ for all $0 < r <1$. Then $\{\mc L_n\}_{n \in \m Z} \cup \{\mc K\}$ defines a representation of the Virasoro algebra with the central charge $c$ if and only if 
\begin{equation}\label{eq:vir_cond}
    L_n^{\m C} \Phi_k (\g)- L_k^{\m C} \Phi_n (\g) = (n - k) \Phi_{n + k} (\g) + \frac{c}{12} (n^3 - n) \delta_{n, -k}.    
\end{equation}
\end{lem}

\begin{proof}
It is clear that $\mc K$ commutes with $\mc L_n$. Let $F : \Loop  \to \m C$ be a smooth cylinder function, then
\begin{align*}
[\mc L_n, \mc L_k] F & = [L_n^{\m C} + \Phi_n,  L_k^{\m C} + \Phi_k] F  \\
& = [L_n^{\m C}, L_k^{\m C}] F +  L_n^{\m C} (\Phi_k  F) +   \Phi_n  L_k^{\m C} F -   L_k^{\m C} (\Phi_n  F) - \Phi_k  L_n^{\m C} F \\
& = (n-k) L_{n +k}^{\m C} F  +  (L_n^{\m C} \Phi_k -  L_k^{\m C} \Phi_n) F,
\end{align*}
where we used Proposition~\ref{p.Stability}. Therefore
   $\{\mc L_n\}_{n \in \m Z}$ satisfies commutation relations of the Virasoro algebra if and only if 
   \begin{align*}
       (n-k) & L_{n +k}^{\m C} F  + (L_n^{\m C} \Phi_k -  L_k^{\m C} \Phi_n) F \\
       & = (n-k) \mc L_{n+k } F + \frac{c}{12} (n^3 - n) \delta_{n, -k} \mc K F \\
  & =  (n-k) (L_{n +k}^{\m C} F  + \Phi_{n+k} F )+  \frac{c}{12} (n^3 - n) \delta_{n, -k} F,
   \end{align*}
   which is equivalent to  \eqref{eq:vir_cond}.
\end{proof}

\begin{prop} \label{prop:Phi_example}
The function
    $$\Phi_{k} = \begin{cases}
    0, \qquad  & k > 0;\\
   \frac{c}{12} P_{ - k} & k \le 0.
\end{cases} 
$$
satisfies \eqref{eq:vir_cond} with central charge $c$.
\end{prop}

\begin{proof}
It was shown in Lemma~\ref{lem:Neretin_property} that $P_{k}$ is bounded on $\Loop^r$ for all $r \in (0,1)$ and $k \ge 0$. 

Let $k \ge 0$. The flows on $\m C$ induced by the vector field $v_k = - z^{k+1} \dd/\dd z$ and $\ii v_k = - \ii z^{k+1} \dd/\dd z$ are given by $\varphi_{k,t}$ and $\psi_{k,t}$ in \eqref{eq:flow}.
For any $R >0$, there exists $t_0 > 0$ such that $\varphi_{k,t}$ is conformal in $B(0,R)$ for all $|t| < t_0$. 
Let $\g_t =\varphi_{k,t} (\g)$ and the associated uniformizing conformal map is given by
$$f_t^{-1} = f^{-1} \circ \varphi_{k, -t}.$$
It is straightforward to compute
$$\frac{\dd}{\dd t}\Big|_{t = 0} \mc S[\varphi_{k, -t}] (z) = (-v_k)'''(z) = (z^{k+1})''' = (k^3 - k) z^{k - 2}.$$
$$\frac{\dd}{\dd t}\Big|_{t = 0} \mc S[\psi_{k, -t}] (z) = (\ii z^{k+1})''' = \ii (k^3 - k) z^{k - 2}.$$

From the chain rule of Schwarzian derivatives, 
\begin{align*}
    \sum_{n \ge 0} L_k P_n z^{n-2} & =  \frac{\dd}{\dd t}\Big|_{t = 0} \mc S[f_t^{-1}] (z)  =  \frac{\dd}{\dd t}\Big|_{t = 0} \mc S[f^{-1} \circ \varphi_{k, -t}] (z) \\
    &= \frac{\dd}{\dd t}\Big|_{t = 0} \left(\mc S[f^{-1}] \circ \varphi_{k, -t} (z)  \left(\varphi_{k, -t}' (z)\right)^2 + \mc S[\varphi_{k, -t}] (z) \right)\\
    & =  \mc S[f^{-1}]'(z) (-v_k (z)) + 2 \mc S[f^{-1}] (z) (-v_k '(z)) + (k^3 - k) z^{k - 2} \\
    & = \left(\sum_{n \ge 0} (n-2) P_n z^{n-3} z^{k+1} + 2  P_n z^{n-2} (k+1) z^{k} \right)  + (k^3 - k) z^{k - 2} \\
    &=   \left(\sum_{n \ge 0} (n+2k) P_n z^{n +k -2}  \right)  + (k^3 - k) z^{k - 2}. 
\end{align*}
Similarly,
\begin{align*}
    \sum_{n \ge 0} L'_k P_n z^{n-2} 
    & =  \mc S[f^{-1}]'(z) (- \ii v_k (z)) + 2 \mc S[f^{-1}] (z) (-\ii v_k '(z)) + \ii  (k^3 - k) z^{k - 2} \\
    &=   \left(\sum_{n \ge 0} \ii (n+2k) P_n z^{n +k -2}  \right)  + \ii (k^3 - k) z^{k - 2}. 
\end{align*}

This shows 
\begin{equation}\label{eq:L_kP_m}
    L_k P_{m} = L_k^{\m C} P_{m}= \begin{cases}
         (k + m ) P_{m-k} + (k^3 - k) \delta_{k,m} \qquad & \text {for $m \ge k \ge 0$} \\
    0 \qquad  &\text {for $k > m \ge 0$}.
    \end{cases}
\end{equation}

Therefore, for $0 \le k \le n$
\begin{align*}
    L_k^{\m C} \Phi_{-n} - L_{-n}^{\m C} \Phi_k & =  \frac{c}{12} L_k^{\m C} P_n  =  \frac{c}{12} (k+n) P_{n - k}  + \frac{c}{12} (k^3-k)\delta_{k,n} \\
    & =   (k+n) \Phi_{k-n}  + \frac{c}{12} (k^3-k)\delta_{k,n}
\end{align*}
which is consistent with \eqref{eq:vir_cond}.

Now we assume $0 \le  n < k$, we check that 
\begin{align*}
    L_{k}^{\m C} \Phi_{-n} - L_{-n}^{\m C} \Phi_k & = 0  =   (k+n) \Phi_{k-n}.
\end{align*}
This covers all the cases where the indices of $L$ and $\Phi$ have different signs.

If both $n, k \ge 0$, then as $\Phi \equiv 0$, we have
\begin{align*}
    L_k^{\m C} \Phi_{n} - L_{n}^{\m C} \Phi_k  = 0 =  (k-n) \Phi_{k+n}.
\end{align*}
It remains to check that 
\begin{align*}
    L_{-k}^{\m C} \Phi_{-n} - L_{-n}^{\m C} \Phi_{-k} =   (n - k) \Phi_{-k-n}
\end{align*}
which is equivalent to 
\begin{equation}\label{eq:P_2_neg}
     L_{-k}^{\m C} P_n - L_{-n}^{\m C} P_k = (n-k) P_{k+n}.
\end{equation}
For this, we use Proposition~\ref{prop:divergence} and \eqref{eq:div_two_f}. For any $F \in \mc C^\infty_c$,
\begin{align*}
    \int F L_{-k}^{\m C}P_n  \dd \mu^c  & = \int \left(- L_{-k}^{\m C} F + \Div_{\mu^c} (L_{-k}^{\m C}) F \right) P_n  \dd \mu^c \\
    & = \int \left(- L_{-k}^{\m C} F - \frac{c}{12} P_k F \right) P_n  \dd \mu^c.  
\end{align*}
Therefore, exchanging $n$ and $k$ and using Proposition~\ref{prop:divergence} again, we obtain
\begin{align*}
   &  \int F (L_{-k}^{\m C}P_n - L_{-n}^{\m C} P_k) \, \dd \mu^c  = \int - L_{-k}^{\m C} F P_n + L_{-n}^{\m C} F P_k \dd \mu^c\\
   & = \frac{12}{c} \int L_{-n}^{\m C} L_{-k}^{\m C} F - L_{-k}^{\m C} L_{-n}^{\m C} F  \dd \mu^c \\
   & =  \frac{12}{c} \int (k -n) L_{-k-n}^{\m C} F  \dd \mu^c.
\end{align*}
This shows
$$\Div_{\mu^c} \left( \frac{12}{c} (k -n) L_{-k-n}^{\m C} \right) = L_{-k}^{\m C}P_n - L_{-n}^{\m C} P_k. $$
On the other hand, we have
$$
\Div_{\mu^c} \left( \frac{12}{c} (k -n) L_{-k-n}^{\m C} \right) = \frac{12}{c} (k -n)  \frac{c}{12} (- P_{k+n})  = (n-k) P_{k+n}
$$
which completes the proof of \eqref{eq:P_2_neg} and shows \eqref{eq:vir_cond}. 
\end{proof}

\begin{cor}\label{cor:vir_gen}
The operator $\mc K$ acting as the identity map on the space of functions on $\Loop$ and the family of operators $\mc L_k: = L_k^{\m C} + \Phi_k$, for $k \in \m Z$, generate the Virasoro algebra of central charge $c$.
\end{cor}
\begin{proof}
    This follows immediately from Lemma~\ref{lem:vir_cond} and Proposition~\ref{prop:Phi_example}.
\end{proof}

\subsection{Unitarizing measure}

We complete the proof of unitarity in this section. For this, we need the following lemma, which shows how $L_k^{\m C}$ intertwines with $\tau$.

\begin{lem}\label{lem:L_tau_comm}
We have $L_k^{\m C} \circ \tau^*  = - \tau^* \circ \overline {L_{-k}^{\m C}}$.
\end{lem}
\begin{proof}
The statement is equivalent to  $L_k (\tau^* F) = -\tau^* (L_{-k} F)$ and $L'_k (\tau^* F) =  \tau^* (L'_{-k} F)$.
 To show this, we note that
    $$\tau \circ \varphi_{k,t} \circ \tau  (z)= \frac{(1+ktz^{-k})^{1/k}}{1/z} = \frac{z}{(1+ (-k)(-t)z^{-k})^{-1/k}} = \varphi_{-k, -t} (z),$$
    $$\tau \circ \psi_{k,t} \circ \tau  (z)= \frac{(1-\ii ktz^{-k})^{1/k}}{1/z} = \frac{z}{(1-\ii(-k)(-t)z^{-k})^{-1/k}} = \psi_{-k, t} (z).$$
 Hence,   
    \begin{align*}
        L_k (\tau^* F) (\g) & = \frac{\dd}{\dd t}\Big|_{t = 0} F \circ \tau \circ \varphi_{k,t} (\g) =  \frac{\dd}{\dd t}\Big|_{t = 0} F \circ \varphi_{-k,-t} \circ \tau (\g) \\
        & = - (L_{-k} F) \circ \tau (\g)  = - \tau^* (L_{-k} F) (\g)
    \end{align*}
    and similarly for $L'_k$ as claimed.
\end{proof}

Recall that we consider $L^{2}$ with the bilinear form defined by \eqref{eq:bilinear}, and we showed that the form is Hermitian and non-degenerate in Proposition~\ref{p.BilinearForm}.

\begin{thm}\label{thm:adjoint}
    The adjoint $\mc L_k^\dag$ of $\mc L_k$ with respect to $(\cdot, \cdot)_{\mu^c}$ is $\mc L_{-k}$ for all $k \in \m Z$. 
\end{thm}

\begin{proof}
Let $k \le 0$, $F, G : \Loop \to \m C$, then
   \begin{align*}
       (F, \mc L_k G)_{\mu^c} & = \int_{\Loop} \overline {F} \tau^* \left(L_k^{\m C} G + \frac{c}{12} P_{-k} G\right) \, \dd \mu^c \\
    \text{\small ``by Lemma~\ref{lem:L_tau_comm}''}   & =\int_{\Loop}  - \overline{F L_{-k}^{\m C}} (\tau^* G)  + \frac{c}{12} \overline {F} \tau^* (P_{-k}  G  )\, \dd \mu^c \\ 
       \text{\small ``by Lemma~\ref{lem:P_Q_tau}''}   & =\int_{\Loop}  - \overline{F L_{-k}^{\m C}} (\tau^* G) + \frac{c}{12} \overline {F Q_{-k}} ( \tau^* G) \, \dd \mu^c \\
        \text{\small ``by \eqref{eq:div_two_f_c}''}   & =\int_{\Loop}  \overline{(L_{-k}^{\m C} F)}  (\tau^* G) - \overline{\Div_{\mu^c} (L_{-k}^{\m C}) F}  (\tau^* G) \\
        & \hspace{40pt} + \frac{c}{12} \overline{F   Q_{-k}}(\tau^* G) \, \dd \mu^c \\
        \text{\small ``by Prop.\,\ref{prop:divergence}''}  & = \int_{\Loop}  \overline{(L_{-k}^{\m C} F)}  (\tau^* G) \, \dd \mu^c \\
        &
        =  (\mc L_{-k} F,  G)_{\mu^c}.
   \end{align*}
    For $k \ge 0$, it suffices to backtrack the proof above. We obtain therefore $\mc L_k^\dag = \mc L_{-k}$ for all $k \in \m Z$.
   \end{proof}

\begin{prop}
The vector fields $\mc L_k$ from Corollary~\ref{cor:vir_gen} are densely defined and closable operators on $L^2 ( \Loop, \mu^{c})$.  
\end{prop}

\begin{proof}
These operators are densely defined by Proposition~\ref{p.Denseness} and by the fact that the multiplication operator is densely defined, as we observed earlier. Now we can use the standard fact in \cite[Thm.\,VIII.1]{ReedSimonI} to see that for these operators to be closable, it is enough to check that $\mc L_k^{\dag}$  are densely defined. But these are operators of the same form by Theorem~\ref{thm:adjoint}, so they are densely defined.  Therefore the closure of  $\mc L_k$ is $\mc L_k^{\dag \dag}$.  
\end{proof}

\vspace{10pt}

\emph{During the final stages of the current article's preparation, we learned that G.~Baverez and A.~Jego were independently working on the Virasoro representations built from the SLE loop measure. They have taken the construction further to develop the CFT for the SLE loop measure and prove many other results, including the uniqueness of Malliavin--Kontsevich--Suhov measure for a given central charge $c \le 1$.  
We encourage the interested reader to have a look at their paper \cite{Baverez_Jego}.}

\subsection*{Acknowledgement}
We thank Peter Lin, Eveliina Peltola, Steffen Rohde, and the participants at the MSRI reading seminar in the Spring 2022 semester for helpful discussions. We are also grateful to Guillaume Baverez and Antoine Jego for communicating with us about their preprint \cite{Baverez_Jego}, which is independent of our work. MG thanks Karl-Hermann Neeb for answering her questions about representation theory. 

This work has been supported by the NSF Grant DMS-1928930 while the authors participated in a program hosted by the Simons Laufer Mathematical Sciences Institute in Berkeley, California, during the Spring 2022 semester.
MG was supported in part by NSF grant DMS-2246549. MG acknowledges the support of the Hausdorff Center of Mathematics (Bonn, Germany) and the IHES (France), where parts of the work were completed.
WQ is partially supported by a GRF grant from the Research Grants Council of the Hong Kong SAR (project CityU11305823). YW was supported by the European Union  (ERC, RaConTeich, 101116694).\footnote{Views and opinions expressed are however those of the authors only and do not necessarily reflect those of the European Union or the European Research Council Executive Agency. Neither the European Union nor the granting authority can be held responsible for them.}

\bibliographystyle{plain}
\bibliography{ref}
\end{document}